\documentclass[12pt, reqno, a4paper]{article}
\usepackage{amsmath,amssymb,amsthm, todonotes,url,mathtools}
\usepackage{todonotes}
\usepackage[lmargin=35mm,rmargin=35mm,bmargin=30mm,tmargin=30mm]{geometry}
\usepackage{bbm}
\usepackage[inline]{enumitem}
\usepackage{tikz}
\usepackage{subcaption}

\usepackage{authblk}

\title{A note on connected greedy edge colouring}
\author[1]{Marthe Bonamy} 
\affil[1]{LaBRI, CNRS,
  Universit\'e de Bordeaux, Bordeaux, France\footnote{M. Bonamy is supported by ANR Project GrR (\textsc{ANR-18-CE40-0032}).}}

\author[2]{Carla Groenland}
\affil[2]{Mathematical  Institute,  University  of  Oxford,  Oxford  OX2  6GG,  United  Kingdom.}

\author[3]{Carole Muller}
\affil{D\'epartement de Math\'ematique, Universit\'e libre de Bruxelles, Belgium
\footnote{C. Muller is supported by the Luxembourg National Research Fund (FNR) Grant Nr. 11628910 and was supported by the FNRS during her research stay in Bordeaux.}}

\author[1]{Jonathan Narboni}
\affil[1]{LaBRI, CNRS,
  Universit\'e de Bordeaux, Bordeaux, France}

\author[4]{Jakub Pek\'arek}
\affil{Computer Science Institute of Charles University, Prague. \footnote{J. Pek\'arek is supported by GAUK grant 118119 (Algorithms for graphs with restrictions on cycles).}}

\author[5]{Alexandra Wesolek}
\affil{Department of Mathematics, Simon Fraser University, Burnaby, BC, Canada\footnote{A. Wesolek is supported by the Vanier Canada Graduate Scholarships program.}}

\date{\today}

\newtheorem{theorem}{Theorem}

\newtheorem{lemma}[theorem]{Lemma}

\newtheorem*{claim*}{Claim}
\newtheorem{problem}{Problem}



\makeatletter
\newtheorem*{rep@theorem}{\rep@title}
\newcommand{\newreptheorem}[2]{
\newenvironment{rep#1}[1]{
 \def\rep@title{#2 \ref{##1}}
 \begin{rep@theorem}}
 {\end{rep@theorem}}}
\makeatother

\newtheorem*{conjecture*}{Conjecture}
\newreptheorem{conjecture}{Conjecture}

\theoremstyle{definition}                    

\theoremstyle{remark}   

\newtheorem*{remark*}{Remark}

\numberwithin{equation}{section}

\usepackage{dsfont}

\tikzstyle{vertex}=[circle, draw, fill=black, inner sep=0pt, minimum width=4pt]

\begin{document}

\maketitle
\begin{abstract}
    Following a given ordering of the edges of a graph $G$, the greedy edge colouring procedure assigns to each edge the smallest available colour. The minimum number of colours thus involved is the chromatic index $\chi'(G)$, and the maximum is the so-called Grundy chromatic index. Here, we are interested in the restricted case where the ordering of the edges builds  the graph in a connected fashion. Let $\chi_c'(G)$ be the minimum number of colours involved following such an ordering. We show that it is NP-hard to determine whether $\chi_c'(G)>\chi'(G)$. We prove that $\chi'(G)=\chi_c'(G)$ if $G$ is bipartite, and that $\chi_c'(G)\leq 4$ if $G$ is subcubic.
\end{abstract}
\section{Introduction}
A naive way to colour the vertices of a graph is to consider them one by one and to colour each vertex with the smallest colour that does not appear on any neighbour of it. 
More formally, let $G$ be a graph and $\mathcal{O} = (v_1,\cdots,v_n)$ be a linear ordering of its vertices. 
The \emph{greedy colouring} of $G$ \emph{following $\mathcal{O}$} is the colouring $\alpha$ of $G$ obtained by colouring $v_i$ with the smallest colour $k$ such that there is no $v_j\in N(v_i)$ with $j<i$ and $\alpha(v_j)= k$, for $i$ from $1$ to $n$. 
The maximum number of colours that can be used using a greedy procedure is called the \emph{Grundy number}, and computing this value can be a convenient way to bound any heuristic used to colour a graph (see \cite{bonnet2018complexity} and \cite{zaker2006results}). 
Finding a good ordering of the vertices can indeed seem like an easier way to find a colouring with not ``too many" colours. 
However, if we choose a bad ordering then the difference between the number of colours involved in a greedy colouring and the chromatic number can be arbitrary large, even for trees.

On the other hand, note that there is always an ordering $\mathcal{O}$ of the vertices of a graph $G$ such that the greedy colouring following $\mathcal{O}$ involves the optimal number of colours, i.e. $\chi(G)$. The argument is simple: consider a $\chi(G)$-colouring $\alpha$ of $G$, and put first all the vertices coloured $1$ in $\alpha$, then all the vertices coloured $2$, etc. The greedy colouring following this ordering might not be exactly the same as $\alpha$, but it will use $\chi(G)$ colours in total. Nevertheless, finding such an ordering is equivalent to directly computing an optimal colouring, so this is not a helpful approach.

A more interesting approach is through connected orderings. A \emph{connected ordering} is an ordering where each vertex (except the first one) has one of its neighbours as predecessor -- in other words, $G[\{v_1,\ldots,v_i\}]$ is connected for every $i$. Note that disconnected graphs do not admit a connected ordering: throughout this paper we only consider connected graphs. Indeed, for colouring purposes, one can simply handle each connected component independently. The minimum number of colours used by the greedy procedure when following a connected ordering is called the \emph{connected chromatic number} of $G$ and is denoted $\chi_c(G)$. Surprisingly, the connected chromatic number behaves similarly to the chromatic number. In fact, Benevides, Campos, Dourado, Griffiths, Morris, Sampaio and Silva \cite{BenevidesCamposDouradoGriffithsMorrisSampaioSilva} proved that $\chi_c(G) \leq \chi(G) +1 $ for every graph $G$.

Edge colouring is a special case of vertex colouring which typically displays significantly meeker behaviour. Indeed, while the chromatic number can vary wildly between the best-known lower bounds and best-known upper bounds, Vizing proved in 1964~\cite{vizing1964estimate} that the number of colours needed to colour the edges of a graph $G$ can only be either $\Delta(G)$ or $\Delta(G)+1$, where $\Delta(G)$ is the maximum degree of $G$. The minimum number of colours needed to colour the edges of a graph $G$ such that any two incident edges have different colours is the \textit{chromatic index} of $G$, denoted $\chi'(G)$. Colouring the edges of a graph $G$ corresponds to colouring the vertices of its so-called \emph{line graph}, where the vertices are $E(G)$ and two vertices are adjacent if the corresponding edges are incident in $G$.

Given that edge colouring is a special case of vertex colouring, all the notions discussed earlier extend naturally. Let us denote by $\chi'_c(G)$ the \emph{connected greedy chromatic index} of $G$. The goal of this paper is to study this parameter. By considering vertex colourings of the line graph of $G$, we obtain $\chi'(G)\leq \chi'_c(G)\leq \chi'(G)+1$. In the case of vertex colouring, it is NP-hard to decide whether $\chi_c(G)=\chi(G)$ \cite{BenevidesCamposDouradoGriffithsMorrisSampaioSilva}. To the best of our knowledge it is unknown whether this extends to edge colouring, and even whether $\chi'(G)$ and $\chi'_c(G)$ can differ. It is however known that $\chi(G)$ and $\chi_c(G)$ can differ on claw-free graphs~\cite{le2018connected}. 

Our first contribution is to prove that deciding $\chi'_c(G)=\chi'(G)$ is NP-hard, even for graphs of small maximum degree satisfying $\chi'(G)=\Delta(G)$.
\begin{theorem}
\label{thm:ge:np}
For all $\Delta \geq 4$, it is NP-hard to decide whether $\chi'(G)=\chi_c'(G)$ on the class of graphs with chromatic index $\Delta$.
\end{theorem}
Our proof also provides an example of a graph $G$ with $\chi'_c(G)>\chi'(G)$ of maximum degree 3. When $G$ is a connected graph of maximum degree 2, then $G$ is a path or a cycle and it is easy to see that $\chi'_c(G)=\chi'(G)$. 

In the vertex case, $2=\chi(G)=\chi_c(G)$ when $G$ is bipartite \cite{BenevidesCamposDouradoGriffithsMorrisSampaioSilva}. We show that for bipartite graphs optimal connected orderings also exist in the edge case.
\begin{theorem}
\label{thm:ge:bipartite}
If $G$ is bipartite, then $\chi_c'(G)=\chi'(G)$.
\end{theorem}
The key argument in the proof of Vizing's theorem is to start from a non-optimal colouring and to reconfigure it to decrease the number of colours used. The reconfiguration step used in the proof, the Kempe change, was first introduced by Kempe in his attempt to prove the four colour theorem and since has become a standard and widely used tool to study colourings as it was proven to often be a fruitful approach.

More formally, an \emph{$(i,j)$-Kempe chain} is a connected component of the subgraph induced by the edges coloured $i$ or $j$, and a \emph{Kempe change} consists of switching the colours of the edges in this component. Note that after switching the colours, the colouring is still proper. Moreover, contrary to the case of vertex colouring, when considering edge colouring, Kempe chains have a much more restrained structure as they can only be paths or even cycles.

In Theorem \ref{thm:ge:bipartite}, we use Kempe changes to reconfigure a $k$-edge colouring to a connected greedy $k$-edge colouring. In order to do this we define the notion of `reachability' which might be of independent interest. 
Let $G'$ be the subgraph induced by the edges of colour $<k$. Reachability predicts whether we can `jump' between the components of $G'$ via a connected ordering that assigns the edges between the components colour $k$; by induction, we can find optimal connected orderings for the components of $G'$, which we combine to an optimal connected ordering for $G$.
We can get a similar reachability result for general graphs (Lemma \ref{lem:ge:reach}), of which the following is an easy corollary.
\begin{theorem}
\label{thm:ge:max_degree3}
If $G$ has maximum degree $3$ then $\chi_c'(G)\leq 4$.
\end{theorem}

However, we did not manage to push through the induction argument used in Theorem \ref{thm:ge:bipartite} to provide a full answer to the following problem, which we leave open.

\begin{problem}[Question 3 in \cite{MotaRochaSilva}]
\label{pr:MRS}
Is it true that $\chi'_c(G)\leq \Delta+1$ for each graph $G$ of maximum degree $\Delta$?
\end{problem}
Throughout this paper, we use the short-cut $xy$ for the edge $\{x,y\}$ and write $[n]$ to mean $\{1,\dots,n\}$. We use the notation $N_G(v)$ for the set of vertices adjacent to $v$ in $G$.

\section{Proof of NP-hardness}
\label{sec:ge:NP}
In this section, we prove Theorem \ref{thm:ge:np}.
 We first define some auxiliary constructions.

Let $\Delta \geq 3$ be given. The $\Delta$-dimensional hypercube $Q_\Delta$ with vertex set $\{0,1\}^\Delta$ is $\Delta$-regular and satisfies $\chi'(Q_\Delta)=\Delta$. Indeed, we may reserve a different colour for each `direction' as in Figure \ref{fig:ge:cube}. 
Pick an edge $xy\in E(Q_\Delta)$. Let $Q_\Delta^+$ be the graph with vertex set $V(Q_\Delta^+)=V(Q_\Delta)\cup \{x',y'\}$ and edge set 
\[
E(Q_\Delta^+)=(E(Q_\Delta)\setminus \{xy\})\cup \{xx',yy'\}.
\]
Then $\chi'(Q_\Delta^+)=\Delta$. An example is given in Figure \ref{fig:ge:cube}. 
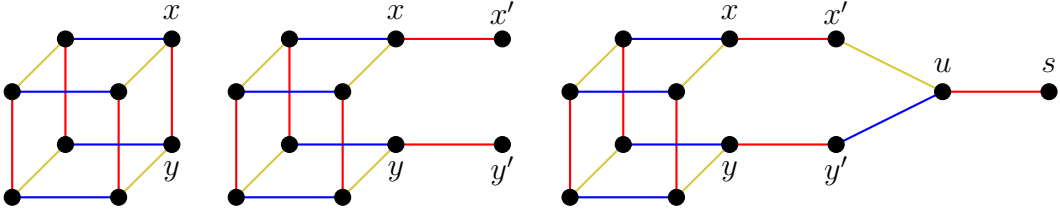
\begin{figure}[htb]
    \centering
    \begin{tikzpicture}[x=0.7cm, y=0.7cm]
        \node[] at (3,3.5){$x$};
        \node[] at (3,.5){$y$};
        
        \draw[thick, red] (0,0)--(0,2) (2,0)--(2,2) (1,1)--(1,3) (3,1)--(3,3);
        \draw[thick, blue] (0,0)--(2,0) (2,2)--(0,2) (1,1)--(3,1) (1,3)--(3,3);
        \draw[thick, yellow!80!black] (0,0)--(1,1) (0,2)--(1,3) (2,0)--(3,1) (2,2)--(3,3);
        
        \draw[ fill=black]  (0,0) circle(3pt);
        \draw[ fill=black]  (1,1) circle(3pt);
        \draw[ fill=black]  (2,0) circle(3pt);
        \draw[ fill=black]  (3,1) circle(3pt);
        \draw[ fill=black]  (2,2) circle(3pt);
        \draw[ fill=black]  (0,2) circle(3pt);
        \draw[ fill=black]  (1,3) circle(3pt);
        \draw[ fill=black]  (3,3) circle(3pt);
        
    \end{tikzpicture}
    \hspace{2mm}
    \begin{tikzpicture}[x=0.7cm, y=0.7cm]
        \node[] at (3,3.5){$x$};
        \node[] at (3,.5){$y$};
        \node[] at (5,3.5){$x'$};
        \node[] at (5,.5){$y'$};
        
        \draw[thick, red] (0,0)--(0,2) (2,0)--(2,2) (1,1)--(1,3)  (3,3)--(5,3) (3,1)--(5,1);
        \draw[thick, blue] (0,0)--(2,0) (2,2)--(0,2) (1,1)--(3,1) (1,3)--(3,3);
        \draw[thick, yellow!80!black] (0,0)--(1,1) (0,2)--(1,3) (2,0)--(3,1) (2,2)--(3,3);
        
        \draw[ fill=black]  (0,0) circle(3pt);
        \draw[ fill=black]  (1,1) circle(3pt);
        \draw[ fill=black]  (2,0) circle(3pt);
        \draw[ fill=black]  (3,1) circle(3pt);
        \draw[ fill=black]  (2,2) circle(3pt);
        \draw[ fill=black]  (0,2) circle(3pt);
        \draw[ fill=black]  (1,3) circle(3pt);
        \draw[ fill=black]  (3,3) circle(3pt);
        \draw[ fill=black]  (5,3) circle(3pt);
        \draw[ fill=black]  (5,1) circle(3pt);
        
    \end{tikzpicture}
    \hspace{2mm}
    \begin{tikzpicture}[x=0.7cm, y=0.7cm]
        \node[] at (3,3.5){$x$};
        \node[] at (3,.5){$y$};
        \node[] at (5,3.5){$x'$};
        \node[] at (5,.5){$y'$};
        \node[] at (7,2.5){$u$};
        \node[] at (9,2.5){$s$};
        
        \draw[thick, red] (0,0)--(0,2) (2,0)--(2,2) (1,1)--(1,3)  (3,3)--(5,3) (3,1)--(5,1) (7,2)--(9,2);
        \draw[thick, blue] (0,0)--(2,0) (2,2)--(0,2) (1,1)--(3,1) (1,3)--(3,3) (5,1)--(7,2);
        \draw[thick, yellow!80!black] (0,0)--(1,1) (0,2)--(1,3) (2,0)--(3,1) (2,2)--(3,3) (5,3)--(7,2);
        
        \draw[ fill=black]  (0,0) circle(3pt);
        \draw[ fill=black]  (1,1) circle(3pt);
        \draw[ fill=black]  (2,0) circle(3pt);
        \draw[ fill=black]  (3,1) circle(3pt);
        \draw[ fill=black]  (2,2) circle(3pt);
        \draw[ fill=black]  (0,2) circle(3pt);
        \draw[ fill=black]  (1,3) circle(3pt);
        \draw[ fill=black]  (3,3) circle(3pt);
        \draw[ fill=black]  (5,3) circle(3pt);
        \draw[ fill=black]  (5,1) circle(3pt);
        \draw[ fill=black]  (7,2) circle(3pt);
        \draw[ fill=black]  (9,2) circle(3pt);
        
    \end{tikzpicture}
    \caption{The graphs $Q_3$, $Q_3^+$ and $H_3$ are depicted with possible $3$-edge colourings.}
    \label{fig:ge:cube}
\end{figure}

Below we consider the situation in which we attempt to extend a colouring in which one of the edges has been precoloured. We assign the lowest available colour to the edges in a connected ordering starting from an edge incident with the precoloured edge.
\begin{lemma}
\label{lem:ge:Qd}
Let $\Delta\geq 3$. Let $xx',yy'\in E(Q_\Delta^+)$ be the two edges containing a vertex of degree 1.
\begin{itemize}
    \item If $\alpha$ is a $\Delta$-edge colouring of $Q_\Delta^+$, then $\alpha(xx')=\alpha(yy')$.
    \item If $xx'$ is precoloured with some colour $i\in [\Delta]$, then there is a connected ordering of the edges of $Q_\Delta^+$ such that the greedy procedure uses $\Delta$ colours.
\end{itemize} 
\end{lemma}
\begin{proof}
To see the first claim, suppose that we assign $xx'$ and $yy'$ different colours. One of the colour classes must then cover an odd number of vertices from $Q_\Delta$ (because it covers an even number of vertices in $Q_\Delta^+$ as any colour class of an edge colouring of $Q_\Delta^+$ forms a matching). Let $v\in V(Q_\Delta)$ be a vertex not covered by this colour class. Since $v$ has degree $\Delta$, there are edges of $\Delta$ different colours incident to it. Hence we have used at least $\Delta+1$ colours. 

To see the second claim, fix any $\Delta$-edge colouring $\alpha$ of $Q_\Delta^+$ with $\alpha(xx')=i$. Let $z\in \{x,x'\}$ be the vertex of degree $\Delta$. We can now always create an ordering of the edges leading to the edge colouring $\alpha$. Indeed, we first colour the edge incident to $z$ which needs to get colour $1$, then the edge incident to $z$ that needs to get colour $2$, etc, until we coloured all edges incident to $z$. We then pick a neighbour of $z$ of degree $\Delta$ and colour all edges incident to this one in a similar order. We continue like this until all edges have been coloured.
\end{proof}

We will extend $Q_\Delta^+$ into a gadget $H$. Let us first explain the case $\Delta=3$. We obtain the graph $H_3$ from the graph $Q_3^+$ by adding a new vertex $u$ adjacent to the vertices $x'$ and $y'$ as well as adding a new vertex $s$ adjacent to $u$ as in Figure \ref{fig:ge:cube}. Suppose we have a connected greedy 3-edge colouring of $H$ starting from $s$. By Lemma \ref{lem:ge:Qd}, $xx'$ and $yy'$ must get the same colour.
Since $x'u$ and $y'u$ cannot get the same colour, the edges $xx',yy'$ and $su$ must all receive the same colour. Since we started from $s$, some edge from $\{xx',yy'\}$ is the first edge to be coloured from $Q_3^+$. Since $x'$ and $y'$ have degree $2$, this edge will not get colour $3$. If we force the edge $su$ to have colour $3$, and then continue in a connected greedy fashion, then this shows we cannot colour all the edges using three colours. On the other hand, if we force it to have colour $1$ or $2$, then we can continue to colour $x'u$, $xx'$, the remainder of the hypercube and finally $yy'$ and $y'u$ using Lemma \ref{lem:ge:Qd}. This proves the lemma below in the case $\Delta =3$.
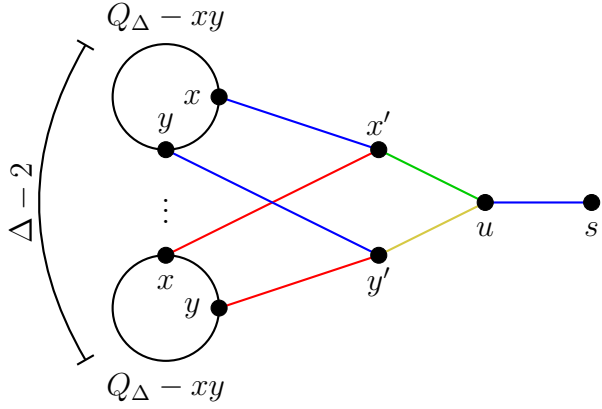
\begin{figure}
    \centering
    \begin{tikzpicture}[x=0.7cm, y=0.7cm]
    \draw[thick, red] (0,1)--(4,3) (1,0)--(4,1);
    \draw[thick, blue] (0,3)--(4,1) (1,4)--(4,3) (6,2)--(8,2);
    \draw[thick, yellow!80!black] (4,1)--(6,2);
    \draw[thick, green!80!black] (4,3)--(6,2);
    
    \draw[thick] (0,0) circle (1);
    \draw[thick] (0,4) circle (1);
    \draw[ fill=black] (0,1) circle (3pt);
    \draw[ fill=black] (1,0) circle (3pt);
    \draw[ fill=black] (0,3) circle (3pt);
    \draw[ fill=black] (1,4) circle (3pt);
    
    \draw[ fill=black] (4,1) circle (3pt);
    \draw[ fill=black] (4,3) circle (3pt);
    \draw[ fill=black] (6,2) circle (3pt);
    \draw[ fill=black] (8,2) circle (3pt);
    
    \node at (0,.5){$x$};
    \node at (.5,0){$y$};
    \node at (0.5,4){$x$};
    \node at (0,3.5){$y$};
    \node at (4,3.5){$x'$};
    \node at (4,.5){$y'$};
    \node at (6,1.5){$u$};
    \node at (8,1.5){$s$};
    
    \node at (0,5.5){$Q_\Delta - xy$};
    \node at (0,-1.5){$Q_\Delta - xy$};
    \node at (0,2){$\vdots$};
    \draw[thick, |-|] (-1.5,-1) to[bend left] (-1.5,5); 
    \node[rotate=90] at (-2.7,2){$\Delta -2$};
    
    \end{tikzpicture}
    \caption{The graph $H$ is depicted with a possible edge colouring.}
    \label{fig:ge:H}
\end{figure}

\begin{lemma}
\label{lem:ge:H}
For any $\Delta \geq 3$, there exists a graph $H$ of maximum degree $\Delta$ with a special vertex $s$ of degree $1$ with the following properties.
\begin{itemize}
    \item If the edge incident with $s$ is precoloured with colour $\Delta$, then there is no connected greedy $\Delta$-edge colouring of $H$ starting from this edge.
    \item If the edge incident with $s$ is precoloured with $i\in [\Delta-1]$, then there exists a connected greedy $\Delta$-edge colouring of $H$ starting from this edge.
\end{itemize}
\end{lemma}
\begin{proof}
We extend $\Delta-2$ copies of $Q_\Delta^+$ to the graph $H$. We first glue all these copies on their respective vertices labelled $x'$ and $y'$. We then obtain the graph $H$ by adding a new vertex $u$ adjacent to the (merged) vertices $x'$ and $y'$ and a new vertex $s$ adjacent to $u$ (see Figure \ref{fig:ge:H}).

Let $\alpha$ be a $\Delta$-edge colouring. Since $\alpha(x'u)\neq \alpha(y'u)$, we find that there exists a $Q_\Delta^+$ copy for which $\alpha(xx')=\alpha(yy')=\alpha(su)$, where $x$ and $y$ are the vertices in this copy adjacent to $x'$ and $y'$ respectively. If we start the colouring from an edge incident to $u$, then one of the edges in $\{xx',yy'\}$ is the first edge to be coloured from $Q_\Delta^+$; since $x'$ has degree $\Delta-1$, this edge will not get colour $\Delta$. Combined with Lemma \ref{lem:ge:Qd}, this shows that no connected $\Delta$-edge colouring starting from $su$ can exist in which the edge $su$ is precoloured $\Delta$. 

On the other hand, if $su$ gets a colour strictly smaller than $\Delta$, then we first may colour $x'u$, then all edges incident to $x'$, and finally by Lemma \ref{lem:ge:Qd} we can further extend the connected ordering in such a way that all copies of $Q_\Delta^+$ are $\Delta$-edge coloured while no edge incident with $y'$ receives colour $\Delta$. So we have at least one colour leftover for $y'u$ (which will in fact need to get colour $\Delta$).
\end{proof}
We are now ready to show that it is NP-hard to decide whether $\chi'(G)=\chi_c(G)$ on the class of graphs of maximum degree $\Delta$, for all $\Delta \geq 4$.
\begin{proof}[Proof of Theorem \ref{thm:ge:np}]
Let $d=\Delta-1$, and let $G$ be an $n$-vertex $d$-regular graph. We transform $G$ into a graph $G'$ of maximum degree $\Delta$ such that $\chi'(G)=d$ if and only if $\chi'_c(G')=\chi'(G')$. In fact, $\chi'(G')=\Delta$ and $|V(G')|\leq \Delta^22^\Delta n$. This reduction proves the theorem since deciding whether $\chi'(G)=d$ is NP-hard on $d$-regular graphs for all $d\geq 3$, as shown by Leven and Galil \cite{Leven83Galil}. 

\begin{figure}
    \centering
    \begin{tikzpicture}[x=0.7cm, y=0.7cm]
    \draw[thick] (-.2, .5)--(1.5,0)--(-.2,-.5) (1.5,0)--(3,0);
    
    \draw[thick, blue] (-2,0) circle (4);
    
    \node at (-2,0){$\vdots$};
    \draw[thick, |-|] (-3,-3) to[bend left] (-3,3); 
    \node[rotate=90] at (-4.5,0){$\Delta -1$};
    
    \node at (-.5,1){$u$};
    \node at (-.5,-1){$u$};
    \node at (-1.5,1.5){$H-s$};
    \node at (-1.5,-1.5){$H-s$};
    \node at (1.5,.5){$s$};
    \node at (3,.5){$v$};
    \node[blue] at (1.5,-3){$G_v$};
    
    \draw[rotate=40, thick] (0,2) ellipse (1 and 1.5);
    \draw[rotate=320, thick] (0,-2) ellipse (1 and 1.5);
    
    \draw[fill=black] (1.5,0) circle(3pt);
    \draw[fill=black] (3,0) circle(3pt);
    \draw[fill=black] (-.2,0.5) circle(3pt);
    \draw[fill=black] (-.2,-0.5) circle(3pt);
    \end{tikzpicture}
    \caption{We create an instance of the depicted graph for each vertex of $G$.}
    \label{fig:ge:Gv}
\end{figure}
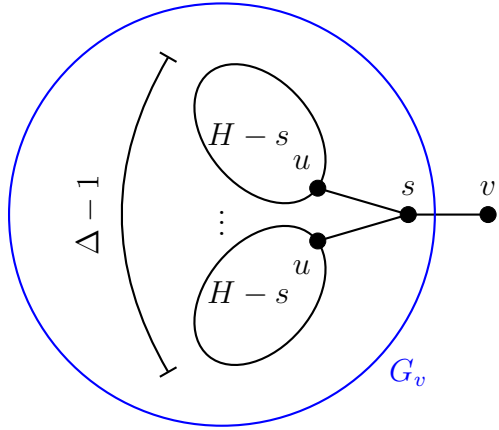
Let $\Delta=d+1$ and let $H$ be the graph from Lemma \ref{lem:ge:H}. For each $v\in V(G)$, we create a graph $G_v$ by merging $\Delta-1$ copies of $H$ on their special vertex $s$ (see Figure \ref{fig:ge:Gv}). The graph $G'$ is obtained from $G$ by connecting $G_v$ to $v$ via an edge for each $v\in V(G)$; for $v,v'$ distinct vertices of $G$, the graphs $G_v$ and $G_{v'}$ are disjoint and have no edges between them. Note that $\chi'(G')=\Delta$. 

Suppose first that our $d$-regular graph $G$ can be coloured using $d$ colours. Fix a $d$-edge colouring $\alpha$ of $G$. There is a connected ordering of the edges of $G$ that results in the edge colouring $\alpha$. Indeed, since $G$ is $d$-regular, whenever we have `reached' a vertex we can assign the edges incident to this vertex the desired colours, starting from the edge coloured $1$, continuing with the edge coloured $2$ etc. We may then colour the edge from $v$ to $G_v$ with colour $d+1$ for all $v\in V(G)$. Continuing in the various copies of $H$, the corresponding edge $su$ gets a colour $<d+1=\Delta$ and hence by Lemma \ref{lem:ge:H} there is a connected ordering in which we can edge colour these with $\Delta$ colours. So $\chi_c'(G')=\chi'(G')$.

Suppose now that $G$ is not $d$-edge colourable. For contradiction, suppose there is a $\Delta$-edge colouring $\alpha$ that can be obtained via a connected ordering. Since $G$ is not $d$-edge colourable, $\alpha(vv')=d+1$ for some $vv'\in E(G)$. 
The two edges between $v,v'$ and $G_v,G_{v'}$ are then not coloured $\Delta$.
As $G_v$ and $G_{v'}$ are not connected to each other, we may assume that these edge are coloured before any of the edges in $G_v$ are coloured.
Since $s$ has degree $\Delta$, there is then a copy of $H$ with vertex $u$ connecting to $s$ in $G_v$ for which $su$ has colour $\Delta$ and this is the first edge of $H$ that is coloured; this contradicts Lemma \ref{lem:ge:H}. So $\chi'_c(G')>\chi'(G')$.
\end{proof}
To obtain a graph $G$ of maximum degree 3 with $\chi_c'(G)>\chi'(G)$, we may take five copies of the graph $H$ from Lemma \ref{lem:ge:H} with $\Delta=3$ and identify their vertices $s$ with the vertices on a $5$-cycle. The resulting graph still has maximum degree 3. Any greedy connected $3$-edge colouring has exactly one edge of colour $3$ on the $5$-cycle; but this means three copies of $H$ have their edge $su$ coloured with colour $3$, and for at least two this is the first edge to be coloured in $H$. This gives a contradiction with  Lemma \ref{lem:ge:H}. Hence $\chi'_c(G)>3=\chi'(G)$.

\section{Bipartite graphs}
\label{sec:ge:bipartite}
Theorem \ref{thm:ge:bipartite} is an immediate consequence of the following lemma.
\begin{lemma}
Let $G$ be a connected bipartite graph with $\chi'(G)\leq k$. Then for any vertex $v\in V(G)$, there exists a connected ordering starting from $v$ leading to a $k$-edge colouring of $G$.
\end{lemma}
\begin{proof}
We prove the lemma by induction on $k$. If $G$ is a connected graph with $\chi'(G)=1$, then $G$ is a single edge. Hence the lemma is true for $k=1$. Suppose now that we have proven the lemma for all $k'<k$ for some integer $k\geq 2$. 

Let $\alpha:E(G)\to [k]$ be a $k$-edge colouring of $G$ and let $u,v\in V(G)$. For $u,v$ distinct, we say $u$ \emph{strongly reaches} $v$ in the colouring $\alpha$ if $uv\in E(G)$ and either $\alpha(uv)<k$ or the degree of $u$ is $k$. Each vertex strongly reaches itself. We now define reachability as the transitive closure of strong reachability: we say $u$ \emph{reaches} $v$ in the colouring $\alpha$ if there is a sequence $u=v_0,v_1\dots,v_\ell=v$ of vertices in $G$ such that $v_{i-1}$ strongly reaches $v_i$ for all $i\in [\ell]$. 

We first show that for every vertex $v$, there exists a $k$-edge colouring of $G$ such that $v$ reaches all vertices of $G$ in this colouring. Take a $k$-edge colouring $\alpha$ of $G$ which maximises the number of vertices that $v$ can reach. Suppose that $v$ cannot yet reach all vertices. We will strictly increase the set of vertices that $v$ can reach through a series of Kempe changes. 

Let $A\subseteq V(G)$ be the set of vertices that $v$ can reach in $\alpha$ and let $B=V(G)\setminus A$. Note that as $v$ reaches itself, $v\in A$. Since $G$ is connected, there must be an edge $su$ from some $s\in A$ to some $u\in B$. By the definition of strong reachability, we find that $s$ has degree strictly smaller than $k$ and that $\alpha(su)=k$. Hence $s$ misses a colour $x\in [k-1]$, that is, it has no edge incident of colour $x$.

Suppose first that $u$ has degree $<k$.
If vertex $u$ misses colour $x$ as well, then the edge $su$ forms a $(k,x)$-component on its own and a $(k,x)$-Kempe change switches the colour of $su$ to $x$. This adds the vertex $u$ to the set of vertices that $v$ can reach, increasing the set of vertices $v$ can reach as desired. Hence we may assume that $u$ misses some colour $y$ but does not miss colour $x$. Then $y<k$ and there is some edge $e$ incident to $u$ coloured $x$. Since all edges between $A$ and $B$ are coloured $k$, the $(x,y)$-component of $e$ stays within $G[B]$. Hence we may perform an $(x,y)$-Kempe change on this component without affecting the set of vertices that $v$ can reach. Now we are back in the case in which $u$ and $s$ both miss colour $x$, which we already handled.

Suppose now that vertex $u$ has degree $k$. Let $e$ denote the edge coloured $x$ incident to $u$. Note that the $(x,k)$-component $C$ of $e$ is a path (of which one endvertex is $s$). If it stays within $G[B]$, then performing an $(x,k)$-Kempe change on $e$ recolours $su$ with colour $x$ without affecting the colours in $G[A]$ and hence strictly enlarges the set of vertices that $v$ can reach. So we may assume that $C$ intersects $A$ a second time, say $s'\in A$ is the vertex closest to $s$ in the path $C$. Since $s'$ has an edge incident with $B$, we find that it has degree $<k$. Hence it has some colour $y<k$ missing. Once we ensure $x$ is missing at $s'$, we can do an $(x,k)$-Kempe change on the component of $e$ and strictly increase the set of vertices that $v$ can reach. 

If $s'$ has an edge incident with colour $x$, then consider the $(x,y)$-component of this edge. This has to stay within $A$ and performing a Kempe change on it will not affect the set of vertices that $v$ can reach since $x,y<k$. The only problem is that this chain $C'$ could include the vertex $s$. Here is where we use that the graph is bipartite: as can be seen in Figure \ref{fig:ge:odd}, this would create an odd cycle in the graph, since there is an odd number of edges in $C$ between $s$ and $s'$ and an even number of edges in $C'$ between $s$ and $s'$ (since they have different colours missing). Hence we may perform the $(x,y)$-switch without affecting the missing colour of $s$, and can then perform the $(x,k)$-switch as desired.
\begin{figure}
    \centering
    \begin{tikzpicture}[x=1.5cm, y=1cm]
    \clip (-3.5,-3) rectangle (3.5,3);
    \draw[] (-5.5,0) ellipse (5 and 3);
    \draw[] (5.5,0) ellipse (5 and 3);
    
    \draw[thick, red] (0.8,-1)--(-.8,-1) (0.8,1)--(-.8,1) (1.5,.5)--(1.5,-.5);
    \draw[thick, blue] (.8,1)--(1.5,.5) (.8,-1)--(1.5,-.5) (-.8,-1)--(-1.5,-.5) (-1.5,.5)--(-2.5,0);
    \draw[thick, green!80!black] (-2.5,0)--(-1.5,-.5) (-1.5,.5)--(-.8,1);
    
    \draw[fill=black] (0.8,-1) circle (3pt);
    \draw[fill=black] (-.8,-1) circle (3pt);
    \draw[fill=black] (0.8,1) circle (3pt);
    \draw[fill=black] (-.8,1) circle (3pt);
    \draw[fill=black] (1.5,.5) circle (3pt);
    \draw[fill=black] (1.5,-.5) circle (3pt);
    \draw[fill=black] (-1.5,.5) circle (3pt);
    \draw[fill=black] (-2.5,0) circle (3pt);
    \draw[fill=black] (-1.5,-.5) circle (3pt);
    
    \node[] at (-.8,1.4){$s$};
    \node[] at (-.8,-1.4){$s'$};
    \node[] at (.8,1.4){$u$};
    
    \node[red] at (0,-1.3){$k$};
    \node[red] at (0,1.3){$k$};
    \node[red] at (1.8,0){$k$};
    
    \node[blue] at (1.2,1){$x$};
    \node[blue] at (1.2,-1){$x$};
    \node[blue] at (-1.2,-1){$x$};
    \node[blue] at (-2.,.5){$x$};
    \node[green!80!black] at (-2.,-.5){$y$};
    \node[green!80!black] at (-1.2,1){$y$};
    \node[green!80!black] at (-.8,-.6){$(y)$};
    \node[blue] at (-.8,.6){$(x)$};

    \node[] at (-1.5,-1.5){$A$};
    \node[] at (2.,-1.5){$B$};
    \end{tikzpicture}
    \caption{If the $(x,y)$-chain of $s'$ includes $s$, then $G$ contains an odd cycle.}
    \label{fig:ge:odd}
\end{figure}
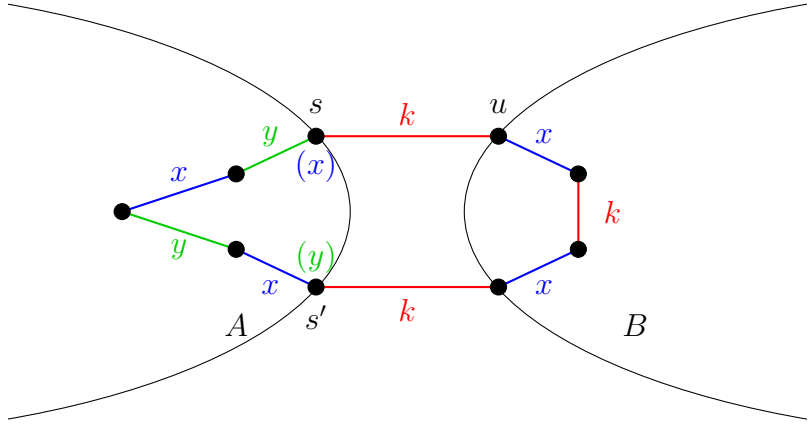

This shows we can always strictly increase the set of vertices that $v$ can reach. This contradicts the maximality of $\alpha$. Hence there exists a colouring $\alpha$ in which $v$ can reach all vertices.

Let $C_1,\dots,C_\ell$ denote the connected components of $G$ when we remove all edges of $G$ coloured $k$ in $\alpha$, where $v\in C_1$. We will show that there is a connected ordering starting from $v$ that leads to a $k$-edge colouring of $G$ which is a  $(k-1)$-edge colouring when restricted to any $C_i$.
Since $v$ can reach everything in $\alpha$, after possibly renumbering $C_2,\dots,C_\ell$, we can find vertices
\[
s_i\in C_1\cup \dots \cup C_i \text{ and }v_{i+1}\in C_{i+1} \cap N_G(s_i),
\]
for $i=1,\dots, \ell$, such that for all $i\in [\ell-1]$, $s_i$ can strongly reach $v_{i+1}$ (`reach in one step') and hence $d_G(s_i)=k$ (since we already know $\alpha(s_iv_{i+1})=k$ by the definition of the components). 

Since $C_1$ is a connected bipartite graph with $\chi'(C_1)\leq k-1$, there exists a connected ordering starting from $v$ that $(k-1)$-edge colours $C_1$ by the induction hypothesis.
By the definition of the components, all edges incident to $s_1$ except for $s_1v_2$ have now been coloured. We colour the edge $s_1v_2$ next; this obtains colour $k$. Since $C_2$ is a connected bipartite graph with $\chi'(C_2)\leq k-1$, there exists a connected ordering starting from $v_2$ that $(k-1)$-edge colours $C_2$ by the induction hypothesis. We extend our previous ordering by this connected ordering and continue like this until we have coloured all edges within the components. We then colour the edges between the components; since colour $k$ will always be available to them, they will all receive a colour at most $k$.
\end{proof}

\section{Subcubic graphs}
\label{sec:ge:max_degree3}
Let $G$ be a graph, let $\alpha$ be a $k$-edge colouring of $G$ and let $i\in [k]$. We say that a vertex $v\in V(G)$ can \emph{$i$-reach} another vertex $w\in V(G)$ in $\alpha$ if there exists a sequence of vertices $v=v_0,v_1,v_2, \dots, v_\ell=w$ of $G$ such that for all $j\in [\ell]$ there is an edge $v_jv_{j+1}\in E(G)$ and one of the following holds:
\begin{itemize}
    \item $\alpha(v_jv_{j+1})<i$;
    \item $v_j$ has incident edges in colours $1,2,\dots,\alpha(v_jv_{j+1})$.
\end{itemize}
If $k=i=\Delta$ the maximum degree of $G$, then this reduces to the notion of reachability from the proof of Theorem \ref{thm:ge:bipartite}. 
The proof of Theorem \ref{thm:ge:max_degree3} follows from the lemma below which might be of independent interest.
\begin{lemma}
\label{lem:ge:reach}
Suppose $G$ is a graph with maximum degree $\Delta$ and $v \in V(G)$. Then $G$ has an $(\Delta+1)$-edge colouring $\alpha$ such that $v$ can $\Delta$-reach all other vertices of $G$ in $\alpha$. 
\end{lemma}
\begin{proof}
In this proof, we will omit the $\Delta$ from $\Delta$-reach.
By Vizing's theorem~\cite{vizing1964estimate}, $G$ has at least one $(\Delta+1)$-edge colouring $\alpha$. We choose such a colouring $\alpha$ that maximises the size of the set $A$ of vertices that $v$ can reach in $\alpha$. Let $B=V(G)\setminus A$ be the set of vertices that $v$ cannot reach.  

Suppose that $A\neq V(G)$. 
The edges between $A$ and $B$ are of colour $\Delta$ or $\Delta+1$. Let $C\subseteq B$ be the neighbours of $A$ via edges coloured $\Delta$. Let $D\subseteq B$ the set of vertices adjacent to a vertex in $C$ (which a priori might overlap with $C$).  We claim that we can obtain an edge colouring in which $v$ can reach a strictly larger set of vertices than $A$ (contradicting the maximality of $A$) as soon as one of the following properties holds.
\begin{itemize}
    \item[(1)] $C$ is empty, i.e. there is no $\Delta$-edge between $A$ and $B$.
    \item[(2)] There is an $(x,\Delta)$-Kempe chain with $x<\Delta$ which is a path between a vertex in $A$ and a vertex in $B$.
    \item[(3)] Some $c\in C$ has a colour $x\in [\Delta-1]$ missing.
    \item[(4)] Some $d\in D$ misses colour $\Delta$ or $\Delta+1$.
\end{itemize}
We will prove the claim after we show that we can assume one of $(1)-(4)$ holds. We suppose all properties above do not hold.
Since (1) fails, we know there is an edge from some $a\in A$ to some $c\in C$ (which has colour $\Delta$ by definition of $C$). Since $c$ is not reachable, there is a colour $x<\Delta$ missing at $a$. Since (3) fails, $c$ is incident to an edge $cd$ of colour $x$. As $(4)$ fails, we know that there is some colour $y<\Delta$ missing at $d$. Consider a $(\Delta, y)$-Kempe chain starting at $d$. Since (2) fails, it stays within $B$. After performing the switch, there is a vertex in $D$ with no edge coloured $\Delta$, contradicting with (4) failing.
 
To prove the claim in case (1), suppose that there are no edges coloured $\Delta$ between $A$ and $B$. Since $G$ is connected, there exists an edge from some $a\in A$ to some $b\in B$. Then $\alpha(ab)=\Delta+1$. Let $x<\Delta+1$ be the smallest colour missing at $a$. Since $b$ has the edge $ab$ incident in colour $\Delta+1$, $b$ misses some colour $y<\Delta+1$. We do an $(x,y)$-Kempe change on the component of $b$ (this could be empty). Since all the edges between $A$ and $B$ are coloured $\Delta+1$, this chain stays within $B$. After the switch both $a$ and $b$ miss the colour $x$. We may recolour the edge $ab$ with colour $x$, and now the set of vertices that $v$ can reach has increased (since $v$ can now reach $b$ as well). 
  
To prove the claim in case (2), suppose that some $(x,\Delta)$-chain for $x<\Delta$ starts in $u\in B$ and contains a vertex $s$ from $A$. Let $a\in A$ and $b\in B$ such that $ab$ is the closest edge between $A$ and $B$ in this chain to $s$. As $x<\Delta$, we find $\alpha(ab)=\Delta$. Thus $a$ must have some colour $y<\Delta$ missing (since $b$ cannot be reached). The $(x,y)$-chain starting at $a$ will stay within $A$ and switching it does not affect which vertices $v$ can reach. So we may assume that $x$ is missing at $a$ and the $(x,\Delta)$-component of $a$ is a path between $a$ and $u$ that only intersects $A$ in the vertex $a$. A Kempe change on this component strictly increases the set of vertices that $v$ can reach. 

We now prove the claim assuming (3). Suppose that $c\in C$ has a colour $x <\Delta$ missing. Let $a\in A$ with $\alpha(ac)=\Delta$ (which exists by the definition of $C$). Let $y<\Delta$ be a colour missing at $a$. The $(x,y)$-chain starting at $c$ stays in $B$, and hence we may perform a switch and then recolour $ac$ to $y$ in order to increase the set of vertices that $v$ can reach. 

Finally, we prove the claim from (4). Suppose $d\in D$ misses colour $\Delta$ or $\Delta+1$. Let $c\in C$ be the vertex $d$ is adjacent to. By (3) we are done unless $c$ only has the colour $\Delta+1$ missing. If $\Delta+1$ is missing at $d$, then we recolour $cd$ to colour $\Delta+1$ in order to reduce to (3). So $\Delta$ is missing at $d$.
Let $a\in A$ with $\alpha(ac)=\Delta$. Let $y=\alpha(cd)<\Delta$ and let $x<\Delta$ be a missing colour at $a$. We may perform an $(x,y)$-Kempe change starting at $a$ to ensure that $a$ misses colour $y$. The only vertices on the $(y,\Delta)$-Kempe chain containing $c$ are then $a$ and $d$. After we switch this chain, the set of vertices that $v$ can reach has strictly increased again.
\end{proof}
We are now ready to prove that any graph of maximum degree $\Delta\leq 3$ satisfies $\chi_c'(G)\leq \Delta+1$.
\begin{proof}[Proof of Theorem \ref{thm:ge:max_degree3}]
Let $G$ be a graph of maximum degree 3. Pick a vertex $v\in V(G)$. Let $\alpha$ be a $4$-edge colouring of $G$ in which $v$ can $3$-reach all other vertices of $G$; this exists by the lemma above.

The proof follows the same argument as the second half of the proof of Theorem \ref{thm:ge:bipartite}, now using the fact that any $(1,2)$-component can be $2$-edge coloured in a connected greedy fashion starting from any vertex instead of applying the induction hypothesis. 

Let $C_1$ be the $(1,2)$-component of $v$. After doing a $(1,2)$-Kempe change if needed, we can colour $C_1$ in a connected greedy fashion starting from $v$. If $G$ has more components, then since $v$ can $3$-reach all other vertices, there must be a $(1,2)$-component $C_2\neq C_1$ and vertices $v_2\in C_2$ and $s_1\in C_1$ such that $s_1v_2 \in E(G)$, and either $\alpha(s_1v_2)<3$ or $s_1$ has incident edges in colours $1,\dots,\alpha(s_1v_2)$ in $\alpha$. Since $s_1$ and $v_2$ are in different $(1,2)$-components, we conclude the latter holds. Since $G$ has maximum degree $3$, it follows that $\alpha(s_1v_2)=3$. Hence all edges incident to $s_1$ have been coloured apart from $s_1v_2$, which we put next in the connected ordering. After performing a $(1,2)$-switch if needed, we $2$-edge colour the edges of $C_2$ in a connected greedy fashion starting from $v_2$. (Note that there might be no edges to colour in this step, as the component might consist of only $v_2$.) As long as the edges of some $(1,2)$-component have not been coloured, we can continue the connected ordering in a similar fashion. The resulting (partial) colouring has the same $(1,2)$-components as $\alpha$ and coloured an edge $3$ if and only if it has colour $3$ in $\alpha$. We finish the connected ordering by first colouring the edges coloured $3$ by $\alpha$ and then the edges coloured $4$ by $\alpha$; all these edges receive a colour at most $4$.
\end{proof}
\paragraph{Acknowledgements} We are grateful to LaBRI for providing us with a great working environment during the time at which the research was conducted. The second and sixth author are thankful for the funding of the ANR Project GrR (\textsc{ANR-18-CE40-0032}) that made their visit possible. 

\bibliographystyle{plain}
\bibliography{greedyedge}

\begin{thebibliography}{1}

\bibitem{BenevidesCamposDouradoGriffithsMorrisSampaioSilva}
Fabr{\'i}cio Benevides, Victor Campos, Mitre Dourado, Simon Griffiths, Robert
  Morris, Leonardo Sampaio, and Ana Silva.
\newblock Connected greedy colourings.
\newblock In Alberto Pardo and Alfredo Viola, editors, {\em LATIN 2014:
  Theoretical Informatics}, pages 433--441, Berlin, Heidelberg, 2014. Springer.

\bibitem{bonnet2018complexity}
{\'E}douard Bonnet, Florent Foucaud, Eun~Jung Kim, and Florian Sikora.
\newblock Complexity of {G}rundy coloring and its variants.
\newblock {\em Discrete Applied Mathematics}, 243:99--114, 2018.

\bibitem{le2018connected}
Ngoc~Khang Le and Nicolas Trotignon.
\newblock Connected greedy colouring in claw-free graphs.
\newblock {\em arXiv preprint arXiv:1805.01953}, 2018.

\bibitem{Leven83Galil}
D.~Leven and Z.~Galil.
\newblock {NP-completeness of finding the chromatic index of regular graphs}.
\newblock {\em Journal of Algorithms}, 4:35--44, 1983.

\bibitem{MotaRochaSilva}
Esdras Mota, Leonardo Rocha, and Ana Silva.
\newblock {Connected greedy coloring of $H$-free graphs}.
\newblock {\em {Discrete Applied Mathematics}}, 284:572--584, 2020.

\bibitem{vizing1964estimate}
Vadim~G. Vizing.
\newblock On an estimate of the chromatic class of a $p$-graph.
\newblock {\em Discret Analiz}, 3:25--30, 1964.

\bibitem{zaker2006results}
Manouchehr Zaker.
\newblock Results on the {G}rundy chromatic number of graphs.
\newblock {\em Discrete mathematics}, 306(23):3166--3173, 2006.

\end{thebibliography}
\end{document}